\documentclass[12pt]{article}
\usepackage{graphicx} 
\usepackage[utf8]{inputenc}

\usepackage{amsmath, amsfonts, amssymb, latexsym, amsthm}
\usepackage{xcolor}
\usepackage{soul}
\usepackage{algorithm}
\usepackage{algpseudocode}
\usepackage{esvect}
\usepackage{graphicx}
\usepackage{setspace}
\usepackage{url}
\usepackage{hyperref}
\usepackage{wrapfig}
\usepackage{tikz}
\usetikzlibrary{calc}
\usetikzlibrary{positioning}
\usepackage{tkz-graph}
\usetikzlibrary{shapes.geometric}
\usetikzlibrary{decorations.markings}
\usetikzlibrary{decorations.pathreplacing}
\usetikzlibrary{patterns}
\usepackage{mathtools}
\theoremstyle{definition}
\newtheorem{theorem}{Theorem}[section]
\newtheorem{lemma}[theorem]{Lemma}

\newtheorem{corollary}[theorem]{Corollary}

\usepackage{comment}

\newcommand{\ccr}{c_{cr}}
\DeclareMathOperator{\DS}{\text{DS}}
\newcommand{\push}{p_{cr}}

\begin{document}

\title{Cops against a cheating robber}
\date{}

\author{
Nancy E.~Clarke\thanks{
Department of Mathematics and Statistics, 
Acadia University, Wolfville, NS, Canada.
Research support by NSERC grant \#2020-06528.
}
\and
Danny Dyer \thanks{
Department of Mathematics and Statistics,
Memorial University of Newfoundland, St.~John's, NL, Canada. Supported by NSERC.
}
\and
William Kellough \footnotemark[2]
}

\maketitle

\begin{abstract}
    We investigate a cheating robot version of Cops and Robber, first introduced by Huggan and Nowakowski, where both the cops and the robber move simultaneously, but the robber is allowed to react to the cops' moves. For conciseness, we refer to this game as Cops and Cheating Robot. The cheating robot number for a graph is the fewest number of cops needed to win on the graph. We introduce a new parameter for this variation, called the push number, which gives the value for the minimum number of cops that move onto the robber's vertex given that there are a cheating robot number of cops on the graph. After producing some elementary results on the push number, we use it to give a relationship between Cops and Cheating Robot and Surrounding Cops and Robbers. We investigate the cheating robot number for planar graphs and give a tight bound for bipartite planar graphs. We show that determining whether a graph has a cheating robot number at most fixed $k$ can be done in polynomial time. We also obtain bounds on the cheating robot number for strong and lexicographic products of graphs. 
\end{abstract}

\begin{quote} \small {\bf Keywords:} discrete time graph processes; pursuit-evasion; cheating robot, Cops and Robber

\smallskip 

\end{quote}

\section{Introduction}\label{sec: intro}

Cops and Robber is a two-player game played on the vertices of a graph. The game begins with a set of cops and a robber being placed on the vertices of a graph. The cops and the robber take turns moving across edges in the graph. If a cop moves to a vertex occupied by the robber, the cops win and the game ends. Otherwise, the robber wins. Cops and Robber was first introduced by Quilliot \cite{Q78} and independently by Nowakowski and Winkler \cite{NW83} where the game was considered with only one cop. Aigner and Fromme \cite{AF84} were the first to study Cops and Robber with multiple cops. The fewest number of cops needed to win on a graph $G$ is called the \emph{cop number} of the graph and is denoted $c(G)$ \cite{AF84}. A survey on Cops and Robber can be found in \cite{BN11}. 

In \cite{HN21}, Huggan and Nowakowski introduced a variant of Cops and Robber where the cops and the robber move simultaneously and the robber is allowed to ``cheat" by knowing in advance how the cops will move each round. \emph{Cheating robot games} are two-player games in which both players move simultaneously, but one player is able to react to the other's move. These types of games are discussed in \cite{H19, HN23}. The player who is able to react to their opponent's moves is referred to as a \emph{cheating robot}. So, we will refer to Huggan and Nowakowski's variant from \cite{HN21} as \emph{Cops and Cheating Robot}. The minimum number of cops needed to win Cops and Cheating Robot on a given graph is called the \emph{cheating robot number}.

In Section \ref{sec: tech intro} we give some preliminary results including an answer to an open question in \cite{HN21}. In Section \ref{sec: push num} we introduce a new parameter called the \emph{push number} for Cops and Cheating Robot and use it to compare the game to Surrounding Cops and Robbers, introduced in~\cite{BCCDFJP20}, another variant of Cops and Robber. The push number is also used in Section \ref{sec: graph products} to obtain upper bounds on the cheating robot number for the strong product of graphs. In Section \ref{sec: planar graphs} we give a tight upper bound on the cheating robot number for bipartite planar graphs. In Section \ref{sec: complexity} we show that determining whether the cheating robot number is at most $k$ on a given graph can be done in polynomial time. We end with directions for future research in Section \ref{sec: further directions}. 

\section{Technical Introduction}\label{sec: tech intro}


Cops and Cheating Robot begins with a set of cops being placed on the vertices of a graph, followed by the robber being placed on a vertex not occupied by a cop. During each turn, the cops and the robber simultaneously either move to a vertex adjacent to the one they currently occupy or stay at the vertex they are occupying. If a cop moves to the vertex occupied by the robber, the cops win. Otherwise, the robber wins. However, while everyone moves at the same time, the robber knows in advance how the cops will move each round. This gives the robber the advantage of being a ``cheating robot" as discussed in Section \ref{sec: intro}. To help distinguish the evader in Cops and Cheating Robot from the evaders of other games we will analyze, we will adopt the nickname ``Robert" for the robber in Cops and Cheating Robot. 

As mentioned in \cite{HN21}, the ruleset of Cops and Cheating Robot can be thought of as turn-based where instead the cops and Robert alternate making moves. When translating Cops and Cheating Robot into turn-based play, the win condition for the cops changes. In this case, the cops win if either Robert ends his turn on a cop vertex or if Robert traverses an edge that was traversed by a cop on her previous move. As was done in \cite{HN21}, whenever we mention Cops and Cheating Robot, we will be referring to this turn-based ruleset.

The \emph{cheating robot number} of a graph $G$, denoted $\ccr(G)$, is the fewest number of cops needed to capture Robert when playing Cops and Cheating Robot on $G$. Huggan and Nowakowski \cite{HN21} proved the following, useful lower bound on the cheating robot number of any graph. 

\begin{theorem}\label{thm: k-core lower bound on ccr}
\cite{HN21} If $G$ is a graph with a $k$-core where $k\in \mathbb{Z}^+$, then $\ccr(G)\geq k$.
\end{theorem}

We note that the difference between the cheating robot number of a graph and the size of its largest $k$-core can be arbitrarily large.

\begin{theorem}\label{thm: ccr - size of largest k-core gets arbitratily large}
    If $N\in \mathbb{Z}^+$ then there exists a graph $G$ where $k$ is the largest integer such that $G$ contains a $k$-core and $\ccr(G) - k > N$.
\end{theorem}

\begin{proof}
    To prove the theorem, we give an infinite family of graphs where the every graph in the family contains a $2$-core, none of the graphs in the family contain a $3$-core and for any $N\in \mathbb{Z}^+$ there exists a graph $G$ in the family such that $\ccr(G) > N$. Consider the family of graphs $\{G_n\}_{n=2}^\infty$ where $G_n$ is obtained by replacing all of the edges in the hypercube $Q_{n-1}$ with 4-cycles. Every vertex of $G_n$ either has degree 2 or degree $2(n-1)$. Furthermore, all of the neighbours of each vertex of degree $2(n-1)$ have degree 2. Thus for all $n\geq 2$, $G_n$ contains a 2-core but not a $k$-core for any $k\geq 3$. Next, we claim that $\ccr(G_n) > n-1$.

    Fix $n\geq 2$ and suppose Robert plays against $n-1$ cops on the graph $G_n$. Since there are $2^{n-1}$ vertices of degree $2(n-1)$, regardless of where the cops place themselves at the start of the game Robert can place himself on a vertex of degree $2(n-1)$. There are $n-1$ other vertices of degree $2(n-1)$ that are each of distance two away from Robert and there are two internally disjoint paths from Robert's vertex to each of the other $n-1$ vertices. Since each of these $n-1$ vertices of degree $2(n-1)$ do not have any neighbours in common, it is not possible for one cop to be adjacent to more than one of these vertices. Therefore Robert can avoid capture indefinitely by waiting at a vertex of degree $2(n-1)$ until a cop moves to his position, and then moving to a vertex of degree $2(n-1)$ before any of the other $n-2$ cops can stop him.
\end{proof}

Huggan and Nowakowski \cite{HN21} gave necessary and sufficient conditions for graphs with a cheating robot number of one.

\begin{theorem}\label{thm: ccr of trees}
\cite{HN21} Let $G$ be a graph. Then $\ccr(G)=1$ if and only if $G$ is a tree.
\end{theorem}

By Theorem \ref{thm: k-core lower bound on ccr}, the minimum degree of a graph gives a lower bound for the cheating robot number. The following result shows that even for cubic graphs with a high enough girth, it is possible for the cheating robot number to be strictly larger than the minimum degree. 

\begin{theorem}
    If $G$ is a graph with $\delta(G) \geq 3$ and girth $g\geq 6$, then $\ccr(G) \geq 4$. 
\end{theorem}

\begin{proof}

    Suppose three cops are able to surround Robert at some vertex. We consider the possibilities for the positions of the cops and of Robert on Robert's last move before he is surrounded. Since we are assuming that every move Robert makes results in him being surrounded, including Robert's option to pass, Robert is on a vertex of degree three and every vertex adjacent to him that is not occupied by a cop has degree three. There are three cases for the position of the cops: none of the cops are within distance one of Robert, exactly one cop is within distance one of Robert and exactly two cops are within distance one of Robert. For convenience, we will say that a cop is \emph{near} Robert if she is at most distance one from him.

    \textbf{No cops near Robert:} Let $v$ be the vertex on which Robert starts on his final move. By moving to an adjacent vertex, Robert is able to avoid being surrounded for an extra turn since no cop can move onto $v$ in one move. This contradicts our assumption that the cops were one move away from surrounding Robert.

    \textbf{One cop near Robert:} There are two possibilities for the cop that is near Robert, either she is adjacent to Robert or she has moved onto Robert's vertex. In both cases, Robert has the same options for his final move and the other two cops are of distance two away from Robert before he moves. If Robert decides to move, the only way the cops can surround him is if the graph contains a cycle of length five, which contradicts our assumption that the graph has girth six.

    \textbf{Two cops near Robert:} If the two cops are adjacent to Robert and Robert moves on his final turn, one of those cops will have to move to a vertex adjacent to Robert that is different from Robert's initial vertex. Thus the graph contains a $C_4$. The same argument holds for the case where one cop is on the vertex occupied by Robert before his final move while the other cop is adjacent to him. If both cops are on the vertex occupied by Robert before his final move then, since Robert's initial vertex has degree three, Robert can move to a vertex that has two neighbours that are not adjacent to these two cops. Thus the cops cannot surround Robert. 
\end{proof}

It is easy to find examples of two graphs $G$ and $H$ such that $H$ is a subgraph of $G$ and $c(H) > c(G)$. For example, for any $n\geq 4$ it holds that $c(C_n) > c(K_n)$. It was posed as an open question in \cite{HN21} whether analogous examples existed for the cheating robot number. Here, we answer this question in the affirmative. 

\begin{theorem}\label{thm: ccr can increase wrt subgraphs deleting edges}
    There exists a graph $G$ with a connected subgraph $H$ such that $\ccr(H) > \ccr(G)$.
\end{theorem}

\begin{proof}
    Consider the graphs $G$ and $H$ shown in Figure \ref{fig: Supergraph counterexample with extra edges}. It is clear that $H$ is a connected subgraph of $G$. 

    Suppose Robert is playing against two cops on the graph $H$ and that he starts on a vertex of degree four. It should be noted that there are two other vertices of degree four, $v_i$ and $v_j$, that are of distance two away from Robert and it is not possible for a single cop to be of distance less than two away from both $v_i$ and $v_j$. Thus, if Robert waits on a vertex of degree four until a cop moves to that vertex then the cops cannot prevent him from moving to another vertex of degree four. So $\ccr(H) > 2$.

    Since $G$ contains a 2-core, $\ccr(G) \geq 2$ by Theorem \ref{thm: k-core lower bound on ccr}. Suppose Robert is playing against two cops, $c_1$ and $c_2$, on $G$. To win the game $c_1$ starts at $v_1$ and $c_2$ starts at $v_3$. If Robert does not start the game on either $v_2$ or $v_4$ then the cops can surround Robert on their first turn. If Robert starts on $v_2$ then the cop $c_2$ can move to $v_2$ which forces Robert to move on his next turn. If Robert then moves closer to $v_1$, he is surrounded. If Robert instead moves further away from $v_1$ then $c_1$ can move to $v_2$ and $c_2$ can move to $v_3$ to surround Robert. Similarly, if Robert starts on $v_4$ the cops can surround him in at most two moves. Therefore $\ccr(G) = 2 < \ccr(H)$.
\end{proof}

\begin{figure}
        \centering
        \begin{tikzpicture}[scale=0.75]
            \tikzstyle{vertex}=[circle, draw=black, minimum size=15pt,inner sep=0pt]

            \node[vertex] at (0,0) (v1) {};
            \node[vertex] at (-2,-1) (v2) {};
            \node[vertex] at (2,-1) (v3) {};
            \node[vertex] at (0,-2) (v4) {};

            \node[vertex] at (1,-3) (v5) {};
            \node[vertex] at (3,-3) (v6) {};
            \node[vertex] at (2,-5) (v7) {};

            \node[vertex] at (0,-4) (v8) {};
            \node[vertex] at (0,-6) (v9) {};
            \node[vertex] at (-2,-5) (v10) {};
            
            \node[vertex] at (-1,-3) (v11) {};
            \node[vertex] at (-3,-3) (v12) {};

            \node at (-3, 0) {\large{$H$}};

            \draw (v1)--(v2)--(v4)--(v3)--(v1);
            \draw (v3)--(v5)--(v7)--(v6)--(v3);
            \draw (v7)--(v8)--(v10)--(v9)--(v7);
            \draw (v10)--(v11)--(v2)--(v12)--(v10);

            \node[vertex, xshift=7cm] at (0,0) (u1) {};
            \node[vertex, xshift=7cm] at (-2,-1) (u2) {$v_1$};
            \node[vertex, xshift=7cm] at (2,-1) (u3) {$v_2$};
            \node[vertex, xshift=7cm] at (0,-2) (u4) {};

            \node[vertex, xshift=7cm] at (1,-3) (u5) {};
            \node[vertex, xshift=7cm] at (3,-3) (u6) {};
            \node[vertex, xshift=7cm] at (2,-5) (u7) {$v_3$};

            \node[vertex, xshift=7cm] at (0,-4) (u8) {};
            \node[vertex, xshift=7cm] at (0,-6) (u9) {};
            \node[vertex, xshift=7cm] at (-2,-5) (u10) {$v_4$};
            
            \node[vertex, xshift=7cm] at (-1,-3) (u11) {};
            \node[vertex, xshift=7cm] at (-3,-3) (u12) {};

            \node[xshift=7cm] at (-3, 0) {\large{$G$}};

            \draw (u1)--(u2)--(u4)--(u3)--(u1);
            \draw (u3)--(u5)--(u7)--(u6)--(u3);
            \draw (u7)--(u8)--(u10)--(u9)--(u7);
            \draw (u10)--(u11)--(u2)--(u12)--(u10);
            \draw (u2)--(u3)--(u7)--(u10)--(u2);
        \end{tikzpicture}
        \caption{A graph $G$ with a cheating robot number of two that contains a subgraph $H$ with fewer edges and a cheating robot number of at least three.}
        \label{fig: Supergraph counterexample with extra edges}
\end{figure}
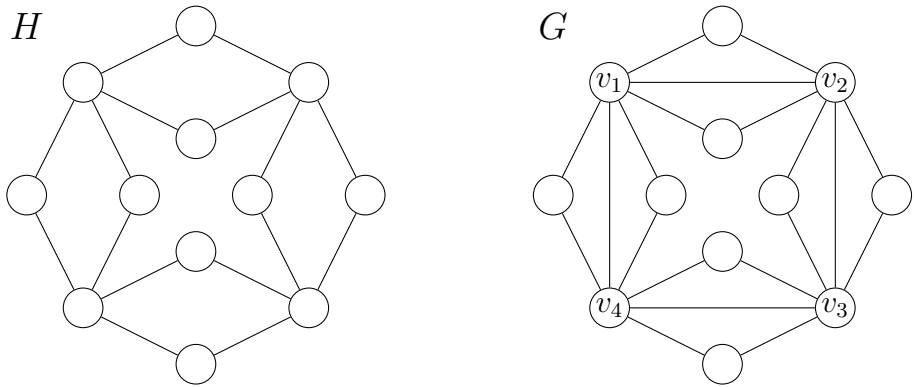

We end our technical introduction by defining the two other Cops and Robber variants that will be mentioned throughout this paper.


Surrounding Cops and Robbers, first introduced in \cite{BCCDFJP20}, is a variant of Cops and Robber with a nearly identical ruleset, with the only exception being the win condition for the cops. In Surrounding Cops and Robbers, the cops win if either they surround the robber or if the robber ends his turn on a vertex occupied by a cop. The \emph{surrounding number} of a graph $G$, denoted $\sigma(G)$, is the fewest number of cops needed to win Surrounding Cops and Robbers on $G$. For results on Surrounding Cops and Robbers and the surrounding number, we direct the reader to \cite{ADG24, BH19, BCCDFJP20, JSU23}. 


Bodyguards and Presidents is a vertex pursuit-evasion game first introduced in \cite{CDK24, K24} where a set of bodyguards is attempting to indefinitely surround a president. The bodyguards and the president alternately take turns, moving the same way as the cops and the robber do in Cops and Cheating Robot and in Surrounding Cops and Robbers. The difference with this game is the win condition for the bodyguards. If, after some finite number of turns, the bodyguards can surround the president and, by the end of every following bodyguard turn, the bodyguards surround the president then the bodyguards win. That is, the bodyguards win if they can surround the president by the end of all but finitely many bodyguard turns. The president wins otherwise. The \emph{bodyguard number} of a graph $G$, denoted $B(G)$, is the fewest number of bodyguards needed to win Bodyguards and Presidents on $G$. 

We end this section by giving results from \cite{CDK24} regarding the bodyguard number. These results will be used in Section \ref{sec: graph products}. For more on Bodyguards and Presidents, we direct the reader to \cite{CDK24, K24}. 

\begin{lemma}\label{lem: Bodyguard numbers of cycles}
    \cite{CDK24} If $n\geq 3$, then
    \begin{equation*}
        B(C_n) = \left\{
        \begin{array}{cc}
             2 & \text{if $n\leq 5$} \\
             3 & \text{if $n>5$.}
        \end{array}
        \right.
    \end{equation*}
\end{lemma}

\begin{lemma}\label{lem: bodyguard num of tree}
    \cite{CDK24} If $G$ is a tree with $\ell$ leaves then $B(G) = \ell$. 
\end{lemma}

\begin{lemma}\label{lem: Bodyguard number of strong grid}
    \cite{CDK24} If $k, n_1, \dots, n_k \in \mathbb{Z}^+$ then $B(\boxtimes^k_{i=1} P_{n_i}) = \frac{3^k - 1}{2}$. 
\end{lemma}

\section{The Push Number}\label{sec: push num}
In this section we define a new parameter that is based on how the cops capture Robert. A \emph{winning strategy} for $k$ cops on a graph $G$ is a pair $(X,A)$ where $X\subseteq V(G)$ is the set of vertices the $k$ cops start the game on and $A$ is an algorithm for the cops' moves such that no matter how Robert moves he is eventually captured.

In order for the cops to capture Robert, they may need to force Robert into a position that is less advantageous for him. To force Robert to move, a cop can move onto the vertex Robert is on. Since ending his turn on a vertex occupied by a cop loses him the game, Robert has no choice but to move to a new vertex. If a cop moves onto Robert's vertex such that Robert does not lose the game by the end of his next turn, then we will say that the cop has \emph{pushed} Robert. Let $G$ be a graph. For a given winning strategy $S$ for $\ccr(G)$ cops, let $\push(G,S)$ denote the maximum number of distinct cops that push Robert over all possible ways Robert could play the game. We define the \emph{push number} of $G$, denoted $\push(G)$, to be the minimum $\push(G,S)$ over all winning strategies for the cops. 

Note that in this definition, we do not care about how many times Robert gets pushed, but only about how many cops push Robert. 

For example, consider Cops and Cheating Robot on the path $P_5$ with $\ccr(P_5) = 1$ cop. If Robert starts on a leaf and stays there, then when the cop moves to the leaf occupied by Robert, every move Robert makes on his next turn will result in him losing. This does not count as a push since a push can only occur if Robert is able to move without losing by the end of his turn. Suppose instead Robert starts on a vertex that is not a leaf and refuses to move off his vertex until a cop pushes him. Since one cop cannot surround Robert from his starting position, she has no choice but to push him. Therefore $\push(P_5) = 1$. Notably, it may occur that the cop pushes Robert more than once, but this does not change the value of the push number. 

Since the push number is always measured when playing with $\ccr(G)$ cops, the number of cops required to push Robert in order to capture him is at most all of the cops in play.

\begin{theorem}\label{thm: push num trivial bd}
    If $G$ is a graph, then $\push(G) \leq \ccr(G)$.
\end{theorem}

The following result is an example of an infinite family of graphs where the cheating robot number of cops can win without pushing Robert.

\begin{theorem}\label{thm: push num of cycles is 0}
    If $n\geq 3$ then $\push(C_n) = 0$.
\end{theorem}

\begin{proof}
    Since $\ccr(C_n) = 2$, suppose Robert is playing against two cops, $c_1$ and $c_2$, on $C_n$ with vertices $v_0, v_1, \dots, v_{n-1}$ where $v_i$ is adjacent to $v_{i+1}$ and the addition is modulo $n$. The cops will begin by placing $c_1$ on $v_0$ and $c_2$ on $v_1$. The cops will move as follows. If $c_1$ is on the vertex $v_i$ and Robert is not on $v_{i-1}$, then $c_1$ will move to $v_{i-1}$. If $c_2$ is on the vertex $v_j$ and Robert is not on $v_{j+1}$, then $c_2$ will move to $v_{j+1}$. By using this strategy, Robert will be surrounded in finitely turns without any pushes. Once he is surrounded, the cops can capture him to end the game without any pushes.
\end{proof}

One simple way for the cops to win without pushing Robert is if they can win extremely quickly. 

\begin{lemma}\label{lem: one move win implies no push}
    Let $G$ be a graph. If $\ccr(G)$ cops can capture Robert in one move on $G$, then $\push(G) = 0$.
\end{lemma}

\begin{proof}
    No pushes occur if Robert is captured in one move.
\end{proof}

Theorem \ref{thm: push num of cycles is 0} gives infinitely many examples where the converse of Lemma \ref{lem: one move win implies no push} holds. However, we can use Lemma \ref{lem: one move win implies no push} to obtain the push number for complete $k$-partite graphs.

\begin{theorem}
    If $n, n_1, n_2, \dots, n_k \geq 2$, then $\push(K_n) = 0$ and $\push(K_{n_1, \dots, n_k}) = 0$.
\end{theorem}

\begin{proof}
    By Lemma \ref{lem: one move win implies no push}, it suffices to show that $\ccr(K_n)$ and $\ccr(K_{m,n})$ cops can win in one move on $K_n$ and $K_{m,n}$ respectively. 

    It is clear that for any graph $G$ on $n$ vertices, $\ccr(G) \leq n-1$. Thus $\ccr(K_n) = n-1$ by Theorem \ref{thm: k-core lower bound on ccr}. Against $n-1$ cops on $K_n$, Robert has no choice but to place himself on a vertex that is already surrounded by cops. Therefore, the cops win in one move.

    Without loss of generality, assume $n_1 \leq n_2 \leq \cdots \leq n_k$ and let $t = \sum_{i=1}^{k-1} n_i$. Since $K_{n_1, \dots, n_k}$ contains a $t$-core, $\ccr(K_{n_1, \dots, n_k}) \geq t$ by Theorem \ref{thm: k-core lower bound on ccr}. For each $1\leq i\leq k$, let $X_i$ be the unique, maximal independent set of vertices in $K_{n_1, \dots, n_k}$ of size $n_i$. The cops can win by placing one cop on each vertex not in $X_{n_k}$. This forces Robert to place himself on a vertex in $X_{n_k}$. However, every vertex in $X_{n_k}$ is surrounded by cops and so the cops can win in one move. 
\end{proof}

Lemma \ref{lem: one move win implies no push} gives a condition that guarantees the push number is zero. The following theorem gives a condition that guarantees the push number is nonzero. 

\begin{theorem}\label{thm: many vertices of high degree implies nonzero push number}
    If $G$ is a graph with at least $\ccr(G) + 1$ vertices of degree at least $\ccr(G) + 1$ then $\push(G) \geq 1$. 
\end{theorem}

\begin{proof}
    Robert can begin the game on a vertex of degree at least $\ccr(G) + 1$ and then wait until a cop is forced to push him. 
\end{proof}

\begin{figure}
    \centering
    \begin{tikzpicture}
        \tikzstyle{vertex}=[circle, draw=black, minimum size=20pt,inner sep=0pt]

        \node[vertex] at (0,0) (v1) {};
        \node[vertex] at (1.5,0) (v2) {};
        \node[vertex] at (3,0) (v3) {};
        \node[vertex] at (3.75,1) (v4) {};
        \node[vertex] at (4.5,0) (v5) {};
        \node[vertex] at (6,0) (v6) {};
        \node[vertex] at (7.5,0) (v7) {};

        \node[yshift = -0.75cm] at (3,0) {$x$};
        \node[yshift = -0.75cm] at (4.5,0) {$y$};

        \draw[thick] (v1)--(v2)--(v3)--(v4)--(v5)--(v6)--(v7);
        \draw[thick] (v3)--(v5);
    \end{tikzpicture}
    \caption{A graph illustrating that the converse of Theorem \ref{thm: many vertices of high degree implies nonzero push number} does not hold.}
    \label{fig: few vertices of high deg doesnt imply low push number}
\end{figure}
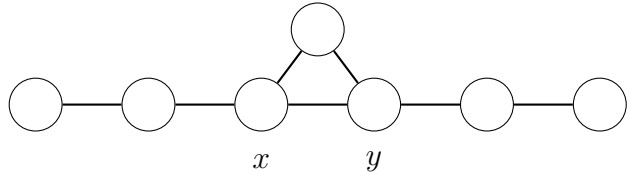

Note that there exist examples of graphs for which the converse of Theorem \ref{thm: many vertices of high degree implies nonzero push number} is not true. Let $G$ be the graph in Figure \ref{fig: few vertices of high deg doesnt imply low push number}. Since $G$ is not a tree, $\ccr(G) \geq 2$. Two cops can win by beginning on the vertices $x$ and $y$ and then taking the shortest path to Robert's vertex until he is captured. So $\ccr(G) = 2$. We claim that regardless of where the cops start, one of them is forced to push Robert. If Robert is able to start on either $x$ or $y$, then the cops cannot surround Robert without pushing him since $\text{deg}(x) = \text{deg}(y) = 3$. If the two cops start on $x$ and $y$, then Robert can start on the vertex that is adjacent to $x$ but not adjacent to $y$. From here, the only way for the cops to win is to push Robert onto the pendent. Therefore $\push(G) > 0$. 

In \cite{HN21}, it is mentioned that for any graph $G$, $\ccr(G) \leq \sigma(G)$. For a proof of this, we direct the reader to \cite{K24}. The push number gives a new way to bound the surrounding number in terms of the cheating robot number.

\begin{theorem}\label{thm: bounds on surr. num in terms of ccr and push num}
    If $G$ is a graph, then $\ccr(G) \leq \sigma(G) \leq \ccr(G) + \push(G).$
\end{theorem}

\begin{proof}
    For convenience, let $k=\ccr(G)$. Since any winning strategy in Surrounding Cops and Robbers is a winning strategy in Cops and Cheating Robot, $\sigma(G) \geq k$. Fix a winning strategy $S$ for the cops in Cops and Cheating Robot such that regardless of how Robert plays, at most $\push(G)$ different cops push Robert before surrounding him.

    We describe a winning strategy for $k + \push(G)$ cops in Surrounding Cops and Robbers as follows. We will label the cops $c_1$, $c_2$, $\dots$, $c_k$, $d_1$, $d_2$, $\dots$, $d_{\push(G)}$. To start the game, the cops $c_1$, $c_2$, $\dots$, $c_k$ will place themselves on the vertices used in the strategy $S$. These cops will play using strategy $S$. Without loss of generality assume that the cops that need to push Robert in order to execute the strategy are the cops $c_1$, $c_2$, $\dots$, $c_{\push(G)}$. For every $1\leq i\leq \push(G)$, the cop $d_i$ will start at the same vertex as $c_i$ and will always move to the vertex occupied by $c_i$ on the previous turn. With this strategy, the robber is unable to traverse an edge previously traversed by $c_1$, $c_2$, $\dots$, $c_{\push(G)}$. Since by assumption these $k$ cops are using a strategy that surrounds Robert in Cops and Cheating Robot, the $k + \push(G)$ cops are able to surround the robber using the same strategy. 
\end{proof}


There exist graphs where the upper bound in Theorem \ref{thm: bounds on surr. num in terms of ccr and push num} is not tight. Let $G$ be the graph in Figure \ref{fig: ccr = sigma(G) but push > 0}. Two cops, $c_1$ and $c_2$, can win Surrounding Cops and Robbers by starting on the vertices shown in Figure \ref{fig: ccr = sigma(G) but push > 0}. If the robber starts on a leaf or on the vertex $y$, he gets surrounded in one move. Suppose the robber starts on the vertex $x$. The cop $c_2$ can move to $x$ while the cop $c_1$ moves to the vertex $c_2$ was just on. This forces the robber to move to either a leaf, where he is surrounded, or to $y$, where the cops can surround him in one additional move. Therefore $\sigma(G) \leq 2$. Since $G$ contains a cycle, by Theorem \ref{thm: ccr of trees} $\ccr(G) > 1$ and so $\ccr(G) = \sigma(G) = 2$. However, by Theorem \ref{thm: many vertices of high degree implies nonzero push number}, $\push(G) \geq 1$ since there are three vertices of degree larger than two. So we have $\sigma(G) < \ccr(G) + \push(G)$.

\begin{figure}[ht]
    \centering
        \begin{tikzpicture}
            \tikzstyle{vertex}=[circle, draw=black, minimum size=25pt,inner sep=0pt]

            \node[vertex] at (0,0) (v1) {$c_1$};
            \node[vertex] at (3,0) (v2) {$c_2$};
            \node[vertex, label = $x$] at (3,3) (v3) {};
            \node[vertex, label = $y$] at (0,3) (v4) {};
            \node[vertex] at (-3,0) (v5) {};
            \node[vertex] at (6,0) (v6) {};
            \node[vertex] at (6,3) (v7) {};

            \draw (v5)--(v1)--(v2)--(v3)--(v4)--(v1);
            \draw (v2)--(v6);
            \draw (v3)--(v7);
        \end{tikzpicture}
    \caption{A graph $G$ such that $\ccr(G) = \sigma(G)$ but $\push(G) > 0$.}
    \label{fig: ccr = sigma(G) but push > 0}
\end{figure}
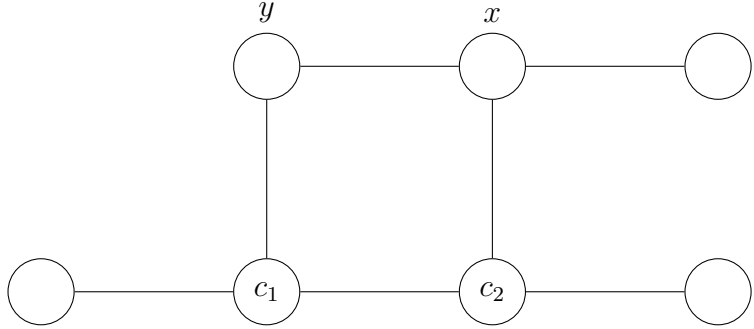

\section{Planar Graphs}\label{sec: planar graphs}

Aigner and Fromme \cite{AF84} showed that for any planar graph $G$, $c(G) \leq 3$. Bradshaw and Hosseini \cite{BH19} gave the following upper bound on the surrounding number for planar graphs. 

\begin{theorem}\label{thm: surr num planar graphs}
    \cite{BH19} Let $G$ be a connected graph. If $G$ is planar, then $\sigma(G) \leq 7$. If $G$ is bipartite and planar, then $\sigma(G) \leq 4$. 
\end{theorem}

As a direct consequence of Theorems \ref{thm: bounds on surr. num in terms of ccr and push num} and \ref{thm: surr num planar graphs}, we have the following.

\begin{corollary}\label{cor: ccr of planar}
    Let $G$ be a connected graph. If $G$ is planar, then $\ccr(G) \leq 7$. If $G$ is bipartite and planar, then $\ccr(G) \leq 4$. 
\end{corollary}

In \cite{CFM24}, the graph operation \emph{double subdivision} on a graph $G$ is defined as follows:
\begin{itemize}
    \item For each $xy\in E(G)$, add a vertex $u$ and add the edges $xu$ and $yu$.
    \item Subdivide every edge $xy\in E(G)$.
\end{itemize}
Alternatively, double subdividing a graph can be thought of as replacing every edge of a graph with a $C_4$. The graph obtained by double subdividing $G$ is denoted $\DS(G)$. 

\begin{lemma}\label{lem: cannot be adj to 3 vertices in G in DS(G)}
    Let $G$ be a graph with vertex set $V$. If $u,v,w\in V \subsetneq V(\DS(G))$, then there does not exist a vertex $x \in V(\DS(G))$ such that $x$ is adjacent to $u$, $v$ and $w$.
\end{lemma}

\begin{proof}
    Suppose $x \in V$. Then in $\DS(G)$, the distance between $x$ and any other vertex in $V$ is at least two. Now suppose $x\notin V$. Then $x$ is adjacent to at most two vertices in $V$. Therefore, $x$ cannot be adjacent to any set of three vertices in $V$. 
\end{proof}

We can use the double subdivision operation to obtain examples of bipartite planar graphs with a cheating robot number of four, proving that one of the bounds in Corollary \ref{cor: ccr of planar} is tight. 

\begin{theorem}\label{thm: ccr=4 bipartite planar existence}
    There exists a planar, bipartite graph with a cheating robot number of four.
\end{theorem}

\begin{proof}
    Consider the double subdivided icosahedron, $\DS(I_{20})$, as illustrated in Figure \ref{fig: DS(I20) bipartite planar graph with ccr = 4}. We will show that three cops are not enough to capture Robert on $\DS(I_{20})$. 

    There are more than three vertices of degree ten, so Robert can begin the game on one of these vertices. Robert's strategy will be to wait at this vertex until a cop moves to his vertex. Note that Robert is of distance two away from five other vertices of degree ten and Robert is able to move closer to any one of them on his next move without traversing an edge that was just traversed by a cop. It is not possible for a single cop to be distance one away from more than two of these vertices by Lemma \ref{lem: cannot be adj to 3 vertices in G in DS(G)}. Therefore, there exists a vertex of degree ten that Robert can move to before any of the cops can. Once Robert has done this, by the symmetry of the graph Robert can repeat the strategy of waiting until he is forced to move indefinitely. 
\end{proof}

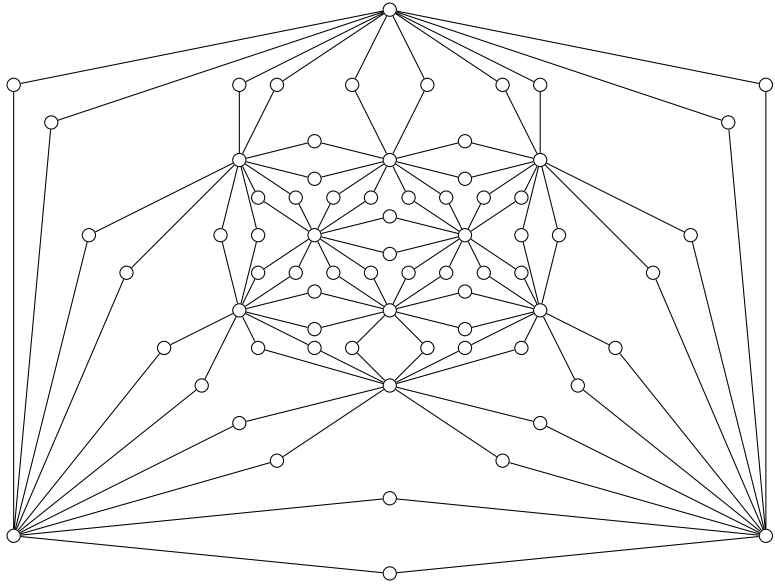
\begin{figure}
        \centering
        \begin{tikzpicture}
        \tikzstyle{vertex}=[circle, draw=black, minimum size=5pt,inner sep=0pt]

        \node[vertex] at (0,1) (v1) {};
        \node[vertex] at (-2,-1) (v2) {};
        \node[vertex] at (0,-1) (v3) {};
        \node[vertex] at (2,-1) (v4) {};
        \node[vertex] at (-1,-2) (v5) {};
        \node[vertex] at (1,-2) (v6) {};
        \node[vertex] at (-2,-3) (v7) {};
        \node[vertex] at (0,-3) (v8) {};
        \node[vertex] at (2,-3) (v9) {};
        \node[vertex] at (0,-4) (v10) {};
        \node[vertex] at (-5,-6) (v11) {};
        \node[vertex] at (5,-6) (v12) {};

        \node[vertex] at (-5, 0) (v1-11L) {};
        \node[vertex] at (-4.5, -0.5) (v1-11R) {};
        \draw (v1)--(v1-11L)--(v11)--(v1-11R)--(v1);

        \node[vertex] at (5, 0) (v1-12L) {};
        \node[vertex] at (4.5, -0.5) (v1-12R) {};
        \draw (v1)--(v1-12L)--(v12)--(v1-12R)--(v1);

        \node[vertex] at (-2, -0) (v1-2L) {};
        \node[vertex] at (-1.5, -0) (v1-2R) {};
        \draw (v1)--(v1-2L)--(v2)--(v1-2R)--(v1);

        \node[vertex] at (-0.5, 0) (v1-3L) {};
        \node[vertex] at (0.5, 0) (v1-3R) {};
        \draw (v1)--(v1-3L)--(v3)--(v1-3R)--(v1);

        \node[vertex] at (1.5, 0) (v1-4L) {};
        \node[vertex] at (2, 0) (v1-4R) {};
        \draw (v1)--(v1-4L)--(v4)--(v1-4R)--(v1);

        \node[vertex] at (-1, -0.75) (v2-3U) {};
        \node[vertex] at (-1, -1.25) (v2-3B) {};
        \draw (v2)--(v2-3U)--(v3)--(v2-3B)--(v2);

        \node[vertex] at (1, -0.75) (v3-4U) {};
        \node[vertex] at (1, -1.25) (v3-4B) {};
        \draw (v3)--(v3-4U)--(v4)--(v3-4B)--(v3);

        \node[vertex] at (-1.75, -1.5) (v2-5L) {};
        \node[vertex] at (-1.25, -1.5) (v2-5R) {};
        \draw (v2)--(v2-5L)--(v5)--(v2-5R)--(v2);

        \node[vertex] at (1.75, -1.5) (v4-6L) {};
        \node[vertex] at (1.25, -1.5) (v4-6R) {};
        \draw (v4)--(v4-6L)--(v6)--(v4-6R)--(v4);

        \node[vertex] at (-0.75, -1.5) (v3-5L) {};
        \node[vertex] at (-0.25, -1.5) (v3-5R) {};
        \draw (v3)--(v3-5L)--(v5)--(v3-5R)--(v3);

        \node[vertex] at (0.25, -1.5) (v3-6L) {};
        \node[vertex] at (0.75, -1.5) (v3-6R) {};
        \draw (v3)--(v3-6L)--(v6)--(v3-6R)--(v3);

        \node[vertex] at (0, -1.75) (v5-6U) {};
        \node[vertex] at (0, -2.25) (v5-6B) {};
        \draw (v5)--(v5-6U)--(v6)--(v5-6B)--(v5);

        \node[vertex] at (-2.25, -2) (v2-7L) {};
        \node[vertex] at (-1.75, -2) (v2-7R) {};
        \draw (v2)--(v2-7L)--(v7)--(v2-7R)--(v2);

        \node[vertex] at (2.25, -2) (v4-9L) {};
        \node[vertex] at (1.75, -2) (v4-9R) {};
        \draw (v4)--(v4-9L)--(v9)--(v4-9R)--(v4);

        \node[vertex] at (-1.75, -2.5) (v5-7L) {};
        \node[vertex] at (-1.25, -2.5) (v5-7R) {};
        \draw (v5)--(v5-7L)--(v7)--(v5-7R)--(v5);

        \node[vertex] at (-0.75, -2.5) (v5-8L) {};
        \node[vertex] at (-0.25, -2.5) (v5-8R) {};
        \draw (v5)--(v5-8L)--(v8)--(v5-8R)--(v5);

        \node[vertex] at (1.75, -2.5) (v6-9R) {};
        \node[vertex] at (1.25, -2.5) (v6-9L) {};
        \draw (v6)--(v6-9L)--(v9)--(v6-9R)--(v6);

        \node[vertex] at (0.75, -2.5) (v6-8R) {};
        \node[vertex] at (0.25, -2.5) (v6-8L) {};
        \draw (v6)--(v6-8L)--(v8)--(v6-8R)--(v6);

        \node[vertex] at (-1, -2.75) (v7-8U) {};
        \node[vertex] at (-1, -3.25) (v7-8B) {};
        \draw (v7)--(v7-8U)--(v8)--(v7-8B)--(v7);

        \node[vertex] at (1, -2.75) (v8-9U) {};
        \node[vertex] at (1, -3.25) (v8-9B) {};
        \draw (v8)--(v8-9U)--(v9)--(v8-9B)--(v8);

        \node[vertex] at (-1.75,-3.5) (v7-10L) {};
        \node[vertex] at (-1,-3.5) (v7-10R) {};
        \draw (v7)--(v7-10L)--(v10)--(v7-10R)--(v7);

        \node[vertex] at (-0.5,-3.5) (v8-10L) {};
        \node[vertex] at (0.5,-3.5) (v8-10R) {};
        \draw (v8)--(v8-10L)--(v10)--(v8-10R)--(v8);

        \node[vertex] at (1.75,-3.5) (v9-10R) {};
        \node[vertex] at (1,-3.5) (v9-10L) {};
        \draw (v9)--(v9-10L)--(v10)--(v9-10R)--(v9);

        \node[vertex] at (-4,-2) (v2-11L) {};
        \node[vertex] at (-3.5,-2.5) (v2-11R) {};
        \draw (v2)--(v2-11L)--(v11)--(v2-11R)--(v2);

        \node[vertex] at (-3,-3.5) (v7-11L) {};
        \node[vertex] at (-2.5,-4) (v7-11R) {};
        \draw (v7)--(v7-11L)--(v11)--(v7-11R)--(v7);

        \node[vertex] at (-2,-4.5) (v10-11L) {};
        \node[vertex] at (-1.5,-5) (v10-11R) {};
        \draw (v10)--(v10-11L)--(v11)--(v10-11R)--(v10);

        \node[vertex] at (4,-2) (v4-12L) {};
        \node[vertex] at (3.5,-2.5) (v4-12R) {};
        \draw (v4)--(v4-12L)--(v12)--(v4-12R)--(v4);

        \node[vertex] at (3,-3.5) (v9-12L) {};
        \node[vertex] at (2.5,-4) (v9-12R) {};
        \draw (v9)--(v9-12L)--(v12)--(v9-12R)--(v9);

        \node[vertex] at (2,-4.5) (v10-12L) {};
        \node[vertex] at (1.5,-5) (v10-12R) {};
        \draw (v10)--(v10-12L)--(v12)--(v10-12R)--(v10);

        \node[vertex] at (0, -5.5) (v11-12U) {};
        \node[vertex] at (0, -6.5) (v11-12B) {};
        \draw (v11)--(v11-12U)--(v12)--(v11-12B)--(v11);
    \end{tikzpicture}
        \caption{The graph $\DS(I_{20})$.}
        \label{fig: DS(I20) bipartite planar graph with ccr = 4}
\end{figure}

\section{Complexity}\label{sec: complexity}

Kinnersley \cite{K15} showed that determining whether a graph has a cop number of at most $k$, where $k$ is not necessarily a fixed integer, is a decision problem that is EXPTIME-complete. For a fixed $k\in \mathbb{Z}^+$, Beraducci and Intrigila \cite{BI93} were the first to show that determining whether $c(G) \leq k$ for a given graph can be done in polynomial time. This result was also proven in \cite{BCP10}. An analogous result, for fixed $k$, holds for the surrounding number as well \cite{BCCDFJP20}. In this section, we will show that determining whether $\ccr(G) \leq k$ for some fixed $k\in \mathbb{Z}^+$ can be done in polynomial time, following \cite{BCP10, BCCDFJP20}. 

\begin{theorem}\label{thm: function used for finding complexity of ccr}
    Let $2^{(V(G),V(G))}$ denote the set of all subsets of $\{(u,v) | u,v\in V(G)\}$. Let $k\in \mathbb{Z}^+$. For a graph $G$, $\ccr(G) > k$ if and only if there exists a function $\psi: (V(\boxtimes^k G), V(\boxtimes^k G)) \to 2^{(V(G), V(G))}$ with the following properties. 
    \begin{itemize}

        \item[(i)] For all $T_1 T_2 \in E(\boxtimes^k G)$, $\psi((T_1,  T_2)) \neq \emptyset$.

        \item[(ii)] For all $T_1 =  (v_1,\dots, v_k), T_2 = (u_1,\dots,u_k) \in V(\boxtimes^k G)$ such that \linebreak $T_1 T_2 \in E(\boxtimes^k G)$,
        \begin{align*}
            \psi((T_1, T_2)) \subseteq \{(r_1, r_2)| &r_1 \in V(G)\backslash (V_{T_1} \cup S_{T_1}), \\
            &r_2 \in V(G)\backslash (V_{T_2} \cup S_{T_2} \cup X_{T_1, T_2, r_1}), \\ &r_1r_2 \in E(G)\}
        \end{align*}
        where $V_{T_1}$ (analogously $V_{T_2}$) is the set of all vertices in the $k$-tuple $T_1$ ($T_2$), $S_{T_1}$ ($S_{T_2}$) is the set of all vertices that can be surrounded by $T_1$ ($T_2$) in one move, and $X_{T_1, T_2, r_1} = \{v_i \in T_1 | r_1=u_i\in T_2\}$.
        
        \item[(iii)] For all $T_1, T_2, T_3 \in V(\boxtimes^kG)$ such that $T_1 T_2, T_2 T_3 \in E(\boxtimes^k G)$, $$\psi_2((T_1, T_2)) \subseteq \psi_1((T_2, T_3))$$ where $\psi_2((T_1, T_2))$ is the set of all second entries of $\psi((T_1, T_2))$ and $\psi_1((T_2, T_3))$ is the set of all first entries of $\psi((T_2, T_3))$.
    \end{itemize}
\end{theorem}

Before proving Theorem \ref{thm: function used for finding complexity of ccr}, here we explain the meaning of Theorem \ref{thm: function used for finding complexity of ccr}. Every vertex of the graph $\boxtimes^k G$ corresponds to a position for $k$ cops on the graph $G$. Two adjacent vertices in $\boxtimes^k G$ corresponds to a legal move that the cops can make on $G$. The purpose of $\psi$ is to map legal cop moves to legal Robert moves that allow him to win the game. Condition (i) ensures that a legal move for Robert exists given any cop move. Condition (ii) ensures that $\psi$ outputs only Robert moves that do not result in him being captured by the cops by the end of their next move. More specifically, condition (ii) has $\psi$ only output Robert moves where he does not start on a vertex occupied by a cop, moves where he does not end his turn on a vertex occupied by a cop, moves where he cannot be surrounded by the cops by the end of their next move, and moves where he is not traversing an edge that the cops traversed. Condition (iii) ensures that for any two consecutive cop moves, Robert can make two consecutive moves that allow him to evade capture. 

\begin{proof}
    For convenience, for two vertices $T_1, T_2 \in V(\boxtimes^k G)$, instead of writing $\psi((T_1, T_2))$ we will write $\psi(T_1 T_2)$. 

    Suppose $\ccr(G) > k$ and Robert is playing against $k$ cops on $G$. For all $T_1, T_2\in V(\boxtimes^k G)$, define $\psi(T_1 T_2)$ to be the set of all ordered pairs of vertices $(r_1, r_2)$ with $r_1, r_2 \in V(G)$ such that if the cops start at $T_1$ and move to $T_2$, then Robert can win by starting on $r_1$ and then moving to $r_2$. Since $\ccr(G) > k$, Robert can always win regardless of how the cops play. Thus $\psi(T_1 T_2) \neq \emptyset$. Furthermore, we know that $r_1$ is not one of the vertices occupied by the cops in $T_1$, the cops cannot surround $r_1$ in one move when moving off of $T_1$, $r_2$ is not one of the vertices occupied by the cops in $T_2$, the cops cannot surround $r_2$ in one move when they move off $T_2$, and as Robert moves from $r_1$ to $r_2$ he does not traverse an edge that the cops traverse moving from $T_1$ to $T_2$. Therefore properties (i) and (ii) hold. 

    Let $T_1, T_2, T_3 \in V(\boxtimes^k)$ such that $T_1$ is adjacent to $T_2$ and $T_2$ is adjacent to $T_3$. That is, $(T_1, T_2)$ and $(T_2, T_3)$ are two consecutive cop moves. Let $(r_1, r_2) \in \psi(T_1 T_2)$. If Robert is on $r_1$ when the cops are on $T_1$, then Robert can win when the cops move to $T_2$ by moving to $r_2$. That is, if the cops move from $T_2$ to $T_3$ Robert can move to some vertex $r_3 \in V(G)$ and continue his winning strategy. So $(r_2, r_3) \in \psi(T_2 T_3)$. Thus property (iii) holds. 

    Now suppose a mapping $\psi$ exists satisfying properties (i), (ii) and (iii). We now construct a winning strategy for Robert against $k$ cops by using $\psi$. Let $T_t \in V(\boxtimes^k G)$ denote the position of the cops in round $t$ of the game. If the cops begin the game at $T_0$, Robert can begin the game on a vertex $r_0$ in the second coordinate of a pair of vertices in $\psi(T_0 T_0)$. We know such a vertex exists since, by property (i), $\psi(T_0 T_0) \neq \emptyset$. By property (ii), $r_0 \notin V_{T_0} \cup S_{T_0}$ and so there is no cop on $r_0$ and the cops cannot surround Robert in one move. Assume that Robert is able to move along edges in $\psi(T_{t-1} T_t)$ for all rounds $t\leq a$ where $a\geq 0$ is fixed. If the cops move from $T_a$ to $T_{a+1}$, then by property (iii) Robert, starting from some vertex $r_a \in \psi_2(T_{a-1} T_a)$, can move to an adjacent vertex $r_{a+1}$ such that $(r_a, r_{a+1}) \in \psi(T_a T_{a+1})$. By property (ii), $r_{a+1} \in V_{T_{a+1}} \cup S_{T_{a+1}} \cup X_{T_a T_{a+1}}$. Thus $r_{a+1}$ does not have a cop on it, the cops are not able to capture Robert in one move from $T_{a+1}$, and Robert is not traversing an edge that the cops traversed while moving from $T_a$ to $T_{a+1}$. So, by induction, Robert can indefinitely avoid capture.
\end{proof}
 
In Algorithm \ref{alg: cheating robot number leq k}, we begin by defining $\psi$ to be a function that maps every cop move to all possible Robert moves that do not immediately lose him the game. Then we delete entries from $\psi_2$ until the conditions of Theorem \ref{thm: function used for finding complexity of ccr} are satisfied. Whenever an entry of $\psi_2$ is deleted, say the $i$th entry, then the $i$th entry of $\psi_1$ will also be deleted. Consequently, if either $\psi_1(T_1 T_2) = \emptyset$ or $\psi_2(T_1 T_2) = \emptyset$, then $\psi(T_1 T_2) = \emptyset$. If by the time Algorithm \ref{alg: cheating robot number leq k} finishes, $\psi(T_1 T_2) = \emptyset$ for some cop move $T_1 T_2 \in E(\boxtimes^k G)$, this tells us that Robert has no safe moves he can make and so $\ccr(G) \leq k$. If, by the time the algorithm finishes, for every cop move $(T_1, T_2)$ where $T_1, T_2 \in V(\boxtimes^k G)$ we have $\psi(T_1 T_2) \neq \emptyset$, then Robert is able to always safely move regardless of how the cops play and so $\ccr(G) > k$. 

\begin{algorithm}
    \caption{CHECK CHEATING ROBOT NUMBER $k$}
    \begin{algorithmic}[1]
        \Require $G=(V,E), k\geq 0$ \\
        initialize $\psi(T_1 T_2)$ to $\{(r_1, r_2) | r_1 \in V(G)\backslash (V_{T_1} \cup S_{T_1}), r_2 \in V(G)\backslash (V_{T_2} \cup S_{T_2} \cup X_{T_1 T_2}), r_1r_2 \in E(G)\}$ for all $T_1, T_2 \in V(\boxtimes^k G)$ \\
        \textbf{repeat} \\
        \ \ \ \textbf{for all} $T_1, T_2, T_3 \in V(\boxtimes^k G)$ such that $T_1 T_2, T_2 T_3 \in E(\boxtimes^k G)$ \textbf{do} \\
        \ \ \ \ \ \ $\psi(T_1 T_2) \gets \psi(T_1 T_2) \cap \{(r_1, r_2) | r_1 \in \psi_1(T_1 T_2), r_2 \in \psi_2(T_1 T_2) \cap \psi_1(T_2 T_3)\}$ \\
        \ \ \ \ \ \ $\psi(T_2 T_3) \gets \psi(T_2 T_3) \cap \{(r_1, r_2) | r_1 \in \psi_1(T_2 T_3) \cap \psi_2(T_1 T_2), r_2 \in \psi_2(T_2 T_3)\}$ \\
        \ \ \ \textbf{end for} \\
        \textbf{until} the value of $\psi$ is unchanged \\ 
        \textbf{if} there exists $T_1, T_2 \in V(\boxtimes^k G)$ such that \\
        \ \ \ $\psi(T_1 T_2) = \emptyset$ \\
        \textbf{then} \\
        \ \ \ \textbf{return} $\ccr(G) \leq k$ \\
        \textbf{else} \\
        \ \ \ \textbf{return} $\ccr(G) > k$ \\
        \textbf{end if}
    \end{algorithmic}
\label{alg: cheating robot number leq k}
\end{algorithm}

\begin{theorem}
    Algorithm \ref{alg: cheating robot number leq k} runs in polynomial time.
\end{theorem}

\begin{proof}
    Note that $|V(\boxtimes^k G)| = n^k$ and $|E(\boxtimes^k G)| \leq \binom{n^k}{2} = O(n^{2k})$. During each iteration of the repeat loop on lines 2--7, $|\psi(T_1 T_2)|$ will decrease by at least one for some $(T_1, T_2)$ where $T_1, T_2, \in V(\boxtimes^k G)$ except for the last iteration of the loop where $\psi$ is unchanged. Since there are $(n^k)(n^k) = n^{2k}$ ways of choosing an ordered pair of vertices in $\boxtimes^k G$, $(T_1, T_2)$, and since $|\psi(T_1, T_2)| = O(n^{2k})$, the repeat loop in lines 2--7 will finish in at most $O(n^{4k})$ steps. 

    Now we consider the number of steps within one iteration of the repeat loop. There are at most $n^{3k}$ ways to choose $T_1, T_2, T_3 \in V(\boxtimes^k G)$ such that $T_1 T_2, T_2 T_3 \in E(\boxtimes^k G)$. There are at most $n^k$ vertices in $\psi_1(T_1, T_2)$ and $\psi_2(T_1, T_2)$ for a given pair of vertices $(T_1, T_2)$. Calculating the intersection between two sets of vertices can be done in $O(n^{2k})$ and calculating the intersection between two sets of ordered pairs of vertices can be done in $O(n^{4k})$ steps. Thus in each iteration of the for loop, there are at most $O(n^{4k}) O(n^{2k}) = O(n^{6k})$ steps. 
    
    Therefore we have a total of $$O(n^{4k}) O(n^{3k}) O(n^{6k}) = O(n^{13k})$$ steps for Algorithm \ref{alg: cheating robot number leq k}. So Algorithm \ref{alg: cheating robot number leq k} finishes in polynomial time.
\end{proof}

\section{Graph Products}\label{sec: graph products}

In this section, we present results for the Cartesian, strong, and lexicographic products of graphs. For two graphs $G$ and $H$, all three products have vertex set $V(G) \times V(H)$.
The \emph{Cartesian} product of $G$ and $H$, denoted $G\square H$, contains the edges $(u,v)(x,y)$ where either $ux \in E(G)$ and $v=y$, or $u=x$ and $vy\in E(H)$. The \emph{strong} product, denoted $G\boxtimes H$, contains the edges $(u,v)(x,y)$ where either $ux \in E(G)$ and $v=y$, $u=x$ and $vy\in E(H)$, or $ux\in E(G)$ and $vy\in E(H)$. The \emph{lexicographic} product, denoted $G\bullet H$, contains the edges $(u,v)(x,y)$ where either $u=x$ and $vy\in E(H)$, or $ux\in E(G)$. For all such products, whether that be $G\square H$, $G\boxtimes H$, or $G\bullet H$, the subgraph $G.\{v\}$ is defined as the graph induced by the set of vertices $\{(u,v) | u\in V(G)\}$. For more on graph products, we direct the reader to the book by Imrich and Klavzar \cite{IK00}.

Huggan and Nowakowski \cite{HN21} showed that for two connected graphs $G$ and $H$, $\ccr(G\square H) \leq \ccr(G) + \ccr(H)$. They also determined the values for $\ccr(P_n \boxtimes P_m)$ and showed that $\ccr(\boxtimes^k_{i=1} P_{n_i}) \leq 3^k$. In this section, we analyze how the cheating robot number behaves with respect to the strong product, which includes determining a tighter upper bound on $\ccr(\boxtimes^k_{i=1} P_{n_i})$. We also consider the cheating robot number of the lexicographic product of two graphs. By making use of the relationship between the cheating robot number and the surrounding number, we obtain new results for the surrounding number of the strong product of graphs. 

We begin with a bound on the cheating robot number for the strong product of any two graphs.

\begin{theorem}\label{thm: General upper bd on ccr of strong prod}
    If $G$ and $H$ are graphs with more than one vertex, then
    \begin{align*}
        \ccr(G\boxtimes H) \leq \min\{ &\ccr(H)\cdot (B(G)+1) + \max\{\push(H), 1\}\cdot \ccr(G), \\ 
        &\ccr(G)\cdot (B(H)+1) + \max\{\push(G), 1\}\cdot \ccr(H)\}.
    \end{align*}
\end{theorem}

\begin{proof}
    Let $S$ be a winning strategy for $\ccr(H)$ cops on $H$ such that at most $\push(H)$ cops push Robert. Let $x_1, \dots, x_{\ccr(H)}$ be the vertices in $H$ that the cops start on when using strategy $S$ and, without loss of generality, suppose the $\push(H)$ cops that Robert can force to push him start on the vertices $x_1, x_2, \dots, x_{\push(H)}$. Let $c_i$ denote the cop starting on $x_i$ for each $1\leq i\leq \ccr(H)$. We will use $S$ to develop a winning strategy for $\ccr(H)\cdot (B(G)+1) + \push(H)\cdot \ccr(G)$ cops on $G\boxtimes H$.

    For a given vertex $v\in V(H)$, we will use the same notation as before where $G.\{v\}$ denotes the subgraph of $G\boxtimes H$ induced by the vertex set $\{(u,v) | u\in V(G)\}$. 
    At the beginning of the game on $G\boxtimes H$, $B(G) + 1$ cops will be placed on $G.\{x_i\}$ for each $1\leq i\leq \ccr(G)$ and an additional $\ccr(G)$ cops will be placed on each $G.\{x_i\}$ for $1\leq i \leq \push(H)$. Let $C_i$ denote the set of all cops that are starting on $G.\{x_i\}$. In finitely many turns, $B(G) + 1$ cops on each $G.\{x_i\}$ will first use a cop winning strategy to catch Robert's shadow on $G.\{x_i\}$. Then while one cop continues moving to remain on Robert's shadow, the other $B(G)$ cops use the bodyguard winning strategy where they treat Robert's shadow as the president. 

    The cops then proceed to capture Robert as follows. If for some $1\leq i\leq \ccr(H)$ the strategy $S$ has the cop $c_i$ move from vertex $v_1$ to $v_2$ where Robert is not on $v_2$ when playing on $H$, then the set of cops $C_i$ will move from $G.\{v_1\}$ to $G.\{v_2\}$ while maintaining the bodyguard winning strategy and while capturing Robert's shadow on $G.\{v_2\}$. The cops are able to accomplish this since, by the construction of $G\boxtimes H$, if a cop is on the vertex $(u, v_1)$ then that cop can move to any vertex in the set $\{(x, v_2)\in V(G\boxtimes H) | x\in N_G [u]\}$. If for some $1\leq i\leq \push(H)$ the strategy $S$ has the cop $c_i$ move from $v_1$ to $v_2$ where Robert is on $v_2$ when playing on $H$, then the $\ccr(G)$ cops from the set $C_i$ that are not a part of the bodyguard winning strategy will move from $G.\{v_1\}$ to $G.\{v_2\}$ while the remaining $B(G)+1$ cops stay on $G.\{v_1\}$ and continue moving onto Robert's shadow and using the bodyguard winning strategy. By using this strategy, the $B(G)+1$ cops prevent Robert from moving onto $G.\{v_1\}$. The $\ccr(G)$ cops will use the cop winning strategy on $G$ to catch Robert on $G.\{v_2\}$ and force Robert to move off $G.\{v_2\}$. Once Robert has moved off of $G.\{v_2\}$, either by force or by his own choice, the $B(G)+1$ cops on $G.\{v_1\}$ will move onto $G.\{v_2\}$ continuing the same strategy of catching Robert's shadow and using the winning bodyguard strategy. 
    
    Since strategy $S$ eventually results in Robert's capture at some vertex, say $v_f$, by using the above strategy on $G\boxtimes H$ the cops are able to force Robert onto $G.\{v_f\}$. Afterwards, $\ccr(G)$ cops will move onto $G.\{v_f\}$ while for every $u\in V(H)$ adjacent to $v_f$, the subgraphs $G.\{u\}$ each contain at least $B(G)+1$ cops that are preventing Robert from moving off of $G.\{v_f\}$. From here the $\ccr(G)$ cops use the cop winning strategy for $G$ to capture Robert on $G.\{v_f\}$.

    This proves that $\ccr(G\boxtimes H) \leq \ccr(H)\cdot (B(G)+1) + \max\{\push(H), 1\}\cdot \ccr(G)$. A similar proof can be done to show that $\ccr(G\boxtimes H) \leq \ccr(G)\cdot (B(H)+1) + \max\{\push(G), 1\}\cdot \ccr(H)$.
\end{proof}

In \cite{HN21}, $\ccr(P_n \boxtimes P_m)$ was determined for any $n,m\in \mathbb{Z}^+$. Next, we obtain exact values for the cheating robot numbers for $C_n \boxtimes P_m$, $C_n \boxtimes C_m$, $K_n \boxtimes P_m$, $K_n \boxtimes C_m$, and $K_n \boxtimes K_m$ for all admissible $n,m\in \mathbb{Z}^+$. 

\begin{theorem}\label{thm: ccr of strong prod exact values}
    \setlength{\abovedisplayskip}{0pt}
    \setlength{\belowdisplayskip}{0pt}
    If $n,m \geq 3$ and $r,s\geq 2$, then
    \begin{gather*}
        \ccr(C_n \boxtimes P_s) = 5, \\
        \ccr(C_n \boxtimes C_m) = 8, \\
        \ccr(K_r \boxtimes P_s) = 2r-1, \\
        \ccr(K_r \boxtimes C_m) = 3r-1,
    \end{gather*}
    and 
    $$\ccr(K_r \boxtimes K_s) = rs - 1.$$
\end{theorem}

\begin{proof}
    We have that $\delta(C_n \boxtimes P_s) = 5$, $\delta(C_n \boxtimes C_m) = 8$, $\delta(K_r \boxtimes P_s) = 2r-1$, $\delta(K_r \boxtimes C_m) = 3r-1$ and $\delta(K_r \boxtimes K_s) = rs-1$. Thus we have that the cheating robot numbers of these products are bounded below by these values by Theorem \ref{thm: k-core lower bound on ccr}. For $K_n\boxtimes P_m$, $K_n\boxtimes C_m$ and $K_n\boxtimes K_m$, the upper bound from Theorem \ref{thm: General upper bd on ccr of strong prod} matches their minimum degrees. For the other two products, $C_n \boxtimes P_s$ and $C_n \boxtimes C_m$, we will describe a winning strategy for the cops. 

    Let $v_1, \dots, v_s$ be the vertices of the path $P_s$ where $v_i$ is adjacent to $v_{i+1}$ for each $1\leq i\leq s-1$. Suppose Robert is playing against five cops on $C_n\boxtimes P_s$. The cops have a winning strategy by starting in the subgraph $C_n.\{v_1\}$. Three of the cops will use the bodyguard winning strategy to surround Robert's shadow on $C_n.\{v_1\}$ while the other two cops use the cop winning strategy to move onto the vertex Robert's shadow occupies. After finitely many moves the closed neighbourhood of Robert's shadow will be occupied by three cops. From here the cops can use a winning strategy from Cops and Cheating Robot to capture Robert on $C_n.\{v_s\}$ as described in the proof for Theorem \ref{thm: General upper bd on ccr of strong prod}.

    If Robert is playing against eight cops on $C_n\boxtimes C_m$ then the cops can set up their strategy by placing four cops on two different cycles, $C_n. \{x_1\}$ and $C_n.\{x_2\}$ where $x_1,x_2 \in V(C_m)$, using the cop winning strategy to move a cop onto Robert's shadow on the two cycles, and then using the bodyguard winning strategy to eventually occupy the closed neighbourhood of Robert's shadow on the two cycles. Six cops will be occupying these closed neighbourhoods which leaves two cops to push Robert where needed and capture him using a winning cheating robot strategy on $C_m$. 
\end{proof}

Next, we consider $k$-dimensional strong grids. 

\begin{theorem}\label{thm: ccr of k-D strong grids}
    If $k\in \mathbb{Z}^+$ and $n_i \geq 3$ for $1\leq i\leq k$, then $$\ccr\left(\boxtimes_{i=1}^k P_{n_i}\right) \leq \sum_{j=0}^{k-1} 3^j = \frac{3^k -1}{2}.$$
\end{theorem}

\begin{proof}
    We proceed by induction on $k$. When $k=2$, $\sum_{j=0}^{k-1} 3^j = 4$. Huggan and Nowakowski \cite{HN21} showed that $\ccr(P_n \boxtimes P_m) \leq 4$ for any $n, m \geq 2$ and so the theorem holds when $k=2$. Fix $k \geq 2$ and assume $$\ccr(\boxtimes_{i=1}^{k - 1} P_{n_i}) \leq \sum_{j=0}^{k-2} 3^j.$$
    We can write $\boxtimes_{i=1}^{k} P_{n_i}$ as $\left(\boxtimes_{i=1}^{k-1} P_{n_i}\right) \boxtimes P_{n_k}$. By applying Theorem \ref{thm: General upper bd on ccr of strong prod} and Lemma \ref{lem: Bodyguard number of strong grid}, we obtain 
    \begin{align*}
        \ccr(\boxtimes_{i=1}^{k-1} P_{n_i} \boxtimes P_{n_k}) &\leq \ccr(P_{n_k})\cdot \left(B(\boxtimes_{i=1}^{k-1} P_{n_i}) + 1\right) + \max\{\push(P_{n_k}), 1\}\cdot \ccr(\boxtimes_{i=1}^{k-1} P_{n_i}) \\
        &\leq (1)\left(3^{k-1} - 1 + 1\right) + (1)\left(\sum_{j=0}^{k-2} 3^j \right) \\
        &= 3^{k-1} + \sum_{j=0}^{k-2} 3^j \\
        &= \sum_{j=0}^{k-1} 3^j
    \end{align*}
    as required. 
\end{proof}

Using a similar technique as in the proof of Theorem \ref{thm: General upper bd on ccr of strong prod}, we obtain an upper bound on the cheating robot number of the lexicographic product of two graphs. 

\begin{theorem}\label{thm: ccr Lexicographic prod upper bound}
    If $G$ and $H$ are graphs, then $$\ccr(G\bullet H) \leq |V(H)|\ccr(G) + \max\{\push(G), 1\}\ccr(H).$$
\end{theorem}

\begin{proof}


    Let $S$ be a winning strategy for $\ccr(G)$ cops on $G$ such that at most $\push(G)$ cops push Robert. Let $x_1\dots, x_{\ccr(G)}$ be the vertices in $G$ that the cops start on when using strategy $S$. Without loss of generality assume the cops that start on the vertices $x_1, \dots, x_{\push(G)}$ are the cops that Robert can force to push him. Let $c_i$ denote the cop starting on $x_i$. 

    Suppose we are playing with $|V(H)|\ccr(G) + \max\{\push(G), 1\}\ccr(H)$ bodyguards on $G\bullet H$. We begin the game by placing $|V(H)|$ cops on each $\{x_i\}.H$ where $1\leq i\leq \ccr(G)$ and an additional $\ccr(H)$ cops on $\{x_j\}.H$ where $1\leq j\leq \push(G)$. The $|V(H)|$ cops on each $\{x_i\}.H$ will distribute themselves so that each vertex in $\{x_i\}.H$ is occupied by exactly one cop from the $|V(H)|$ cops while the $\ccr(H)$ cops can be placed anywhere within each $\{x_i\}.H$. Let $C_i$ denote the set of all cops starting on $\{x_i\}.H$. Once the cops are set up, if for some $1\leq i\leq \ccr(G)$ the strategy $S$ has the cop $c_i$ move from the vertex $v_1$ to $v_2$ where Robert is not on $v_2$ when playing on $G$, then the set of cops $C_i$ will move from $\{v_1\}.H$ to $\{v_2\}.H$ such that every vertex of $\{v_2\}.H$ ends up with a cop on it. If, for some $1\leq i\leq \push(G)$, the strategy $S$ has the cop $c_i$ move from $v_1$ to $v_2$ where Robert is on $v_2$ when playing on $G$, then the $\ccr(H)$ cops from $C_i$ will move from $\{v_1\}.H$ to $\{v_2\}.H$ while the $|V(H)|$ cops on $\{v_1\}.H$ remain at $\{v_1\}.H$. Next, the $\ccr(H)$ use a winning strategy to capture Robert on $\{v_2\}.H$, forcing him to another copy of $H$. Once Robert has moved off $\{v_2\}.H$, the $|V(H)|$ cops on $\{v_1\}.H$ move to $\{v_2\}.H$.

    Since strategy $S$ results in Robert's capture on $G$, the above strategy will trap Robert onto some copy of $H$, say $\{v_f\}.H$. From there, a group of $\ccr(H)$ cops can move onto $\{v_f\}.H$ and capture Robert. 
\end{proof}

By considering Theorem \ref{thm: ccr Lexicographic prod upper bound}, we can obtain the exact value of the cheating robot numbers for $P_n \bullet P_m$, $P_n \bullet C_m$, $C_n \bullet P_m$, and $C_n \bullet C_m$.

\begin{corollary}\label{cor: Lexicographic prod equalities}
    \setlength{\abovedisplayskip}{0pt}
    \setlength{\belowdisplayskip}{0pt}
    If $n,m, r, s\in \mathbb{Z}^+$ and $r, s \geq 2$, then
    \begin{gather*}
        \ccr(P_n\bullet P_m) = m+1, \\
        \ccr(P_n \bullet C_s) = s+2, \\
        \ccr(C_r \bullet P_m) = 2m+1, \\
        \ccr(C_r \bullet C_s) = 2s+2.
    \end{gather*}
\end{corollary}

By using the bodyguard number, an analogous result to Theorem \ref{thm: General upper bd on ccr of strong prod} for the surrounding number can be obtained. 

\begin{theorem}\label{thm: general upper bound on surr num for strong products}
    If $G$ and $H$ are graphs, then
    \begin{align*}
        \sigma(G\boxtimes H) \leq \min\{ &\sigma(H)\cdot (B(G)+1) + \sigma(G), \\
        &\sigma(G)\cdot (B(H)+1) + \sigma(H)\}.
    \end{align*}
\end{theorem}

\begin{proof}
    Let $S$ be a winning strategy in Surrounding Cops and Robbers for $\sigma(H)$ cops on $H$ and let $x_1,\dots, x_{\sigma(H)}$ be the vertices in $H$ that the cops start on when using strategy $S$. We now give a winning strategy for $\sigma(H)\cdot (B(G)+1) + \sigma(G)$ cops in Surrounding Cops and Robbers on the graph $G\boxtimes H$.

    At the beginning of the game, $B(G) + 1$ cops will be placed on $G.\{x_i\}$ for each $1\leq i\leq \sigma(H)$ and an additional $\sigma(G) + 1$ cops will be placed on $G.\{x_1\}$. In finitely many turns, the $B(G) + 1$ cops on each $G.\{x_i\}$ will use the cop winning strategy and the bodyguard strategy to indefinitely capture the closed neighbourhood of the robber's shadow in the same way as in the proof of Theorem \ref{thm: General upper bd on ccr of strong prod}. These cops will also move to adjacent copies of $G$ as needed according to the strategy $S$ in the same way as in the proof of Theorem \ref{thm: General upper bd on ccr of strong prod}. In finitely many moves, the robber will be forced onto the subgraph $G.\{v_f\}$ and for each $u\in V(H)$ such that $uv_f \in E(H)$, the robber will be unable to move onto $G.\{u\}$ due to the $B(G) + 1$ cops on each of the $G.\{u\}$. To finish the game, the remaining $\sigma(G)$ cops will move to $G.\{v_f\}$ and use a winning strategy in Surrounding Cops and Robbers to eventually surround the robber in $G.\{v_f\}$. Since the rest of the robber's adjacent vertices outside of $G.\{v_f\}$ are occupied by the groups of $B(G) + 1$ cops, the robber is surrounded and the cops win. 
    
    This proves that $\sigma(G\boxtimes H) \leq \sigma(H)\cdot (B(G)+1) + \sigma(G)$. A similar proof can be done to show that $\sigma(G\boxtimes H) \leq \sigma(G)\cdot (B(H)+1) + \sigma(H)$.
\end{proof}


The bound in Theorem \ref{thm: general upper bound on surr num for strong products} is tight. Consider the game being played on $C_n \boxtimes C_m$ where $n\leq 5$. Since $\delta(C_n \boxtimes C_m) = 8$, $\sigma(C_n \boxtimes C_m) \geq 8$. It is easy to see that $\sigma(C_n) = \sigma(C_m) = 2$. From Lemma \ref{lem: Bodyguard numbers of cycles} we know that $B(C_n) = 2$. Therefore the upper bound is $2(2+1) + 2 = 8$. Thus, $\sigma(C_n \boxtimes C_m) = 8$.

\begin{theorem}\label{thm: surr num of k-D strong grids}
    If $k\in \mathbb{Z}^+$ and $n_i \geq 3$ for $1\leq i\leq k$, then $$\sigma\left(\boxtimes^k_{i=1} P_{n_i}\right) \leq \frac{3^k + 1}{2}.$$
\end{theorem}

Notably, Theorem \ref{thm: surr num of k-D strong grids} cannot be proven using the same induction technique from the proof of Theorem \ref{thm: ccr of k-D strong grids}. Indeed, if we attempted to use induction on $k$ and applied Theorem \ref{thm: general upper bound on surr num for strong products}, in the induction step we would obtain 
\begin{equation*}
    \sigma(\boxtimes^k_{i=1} P_{n_i}) \leq \min\left\{ \frac{5(3^{k-1}) + 1}{2}, \frac{3^k + 7}{2}\right\}.
\end{equation*}
Instead, we make use of the push number and Theorem \ref{thm: bounds on surr. num in terms of ccr and push num}.

\begin{proof}
    From Theorem \ref{thm: ccr of k-D strong grids}, $\ccr(\boxtimes^k_{i=1} P_{n_i}) \leq \frac{3^k -1}{2}$. To show that $\sigma(\boxtimes^k_{n_i}) \leq \frac{3^k -1}{2} + 1$, by Theorem \ref{thm: bounds on surr. num in terms of ccr and push num} it suffices to show that $\frac{3^k - 1}{2}$ cops can capture Robert on $\boxtimes^k_{i=1} P_{n_i}$ with only one cop pushing Robert. In this proof, we will allow Robert to be more powerful, as if he were a robber in the surrounding variant, for part of the game and then only have one cop push him afterwards.


    Begin by placing all of the cops on the vertex $(0,\dots, 0)$. Let $(v_1, \dots, v_k)$ denote the vertex Robert is on. 
    For $2\leq \ell\leq k$, let $A_\ell$ denote the set of all vertices adjacent to Robert labelled $(x_1, \dots, x_{\ell-1}, v_\ell - 1, v_{\ell+1}, v_{\ell+2}, \dots, v_k)$ where $x_i \in \{v_i - 1, v_i, v_i + 1\}$. For each $2 \leq \ell \leq k$, $|A_\ell| = 3^{\ell - 1}$. Thus, $\sum_{\ell=2}^k |A_\ell| = \frac{3^k - 3}{2}$. Since $c(\boxtimes^k_{i=1} P^\prime) = 1$, $\frac{3^k - 3}{2}$ cops can place themselves on the vertices in $A_2, \dots, A_k$ in finitely many turns. While these cops are getting into position, we will allow Robert to have the ability to traverse edges that cops traverse as if Robert was a robber in the surrounding variant. Regardless of whether he has this ability or not, it does not change the cops' ability to set up their strategy. This leaves one cop who will, after finitely many turns, move onto the vertex Robert is on and continue to push him every turn. Once all of the cops are in their positions, there are three possibilities every time Robert moves from $(v_1,\dots, v_k)$ to $(v^\prime_1, \dots, v^\prime_k)$:
    \begin{itemize}
        \item[(i)] $v^\prime_1 = v_1 + 1$ and $v^\prime_i = v_i$ for all $2\leq i\leq k$;
        \item[(ii)] $v^\prime_1 = v_1 - 1$ and $v^\prime_i = v_i$ for all $2\leq i\leq k$;
        \item[(iii)] for some $2\leq i\leq k$, $v^\prime_i \neq v_i$.
    \end{itemize}

    We claim that if (iii) occurs enough times, Robert will be forced onto a vertex of the form $(x, n_2,\dots, n_k)$ where $1\leq x\leq n_1$. Because of the cops' placements, we claim that Robert is not able to move back to any vertex he was on in any previous round. Suppose for a contradiction that Robert was able to move to a vertex $(u_1,\dots, u_k)$ he was at on some previous round. By the way the cops are positioned, at least one of Robert's coordinates increases whenever he moves. Let $u_{j_1}, \dots, u_{j_m}$ with $j_1\leq j_2\leq \cdots \leq j_m$ be the coordinates of Robert's position that increase when Robert first moves off of $(u_1,\dots, u_k)$. Then for Robert to be able to move back onto $(u_1, \dots, u_k)$, at some point he would need to decrease $u_{j_m}$ without increasing any of $u_{j_m + 1},\dots, u_k$ which is not possible due to the positioning of the cops. Therefore every time (iii) occurs, Robert moves to a vertex that he can never go back to. So if (iii) occurs at most $\prod^k_{i=2} n_i$ times, he will be forced onto a vertex of the form $(x, n_2,\dots, n_k)$.

    Next, we claim that (iii) is impossible for Robert to avoid indefinitely. Suppose for a contradiction that he can. If (i) occurs, since Robert is unable to move back to the vertex he was on in the previous round, he is unable to move such that (ii) occurs. Thus, (i) is Robert's only movement option if he wants to avoid (iii). However, Robert's first coordinate $v_1$ can only decrease finitely many times until $v_1 = 1$. Once this occurs, on Robert's next move (iii) will occur. A similar argument holds if (ii) occurs. Therefore, (iii) is impossible for Robert to avoid indefinitely and, for as long as the game continues, Robert will continually be forced to change at least one of $v_2, \dots, v_k$ after finitely many turns.

    So after finitely many turns, Robert will be on the path induced by the vertices $(1, n_2,\dots, n_k)$, $(2, n_2,\dots, n_k)$, \dots, $(n_1, n_2,\dots, n_k)$. Once here, Robert will be unable to move off the path due to the cops' positioning. Eventually the one cop that is continuously pushing Robert will force him onto either $(1, n_2,\dots, n_k)$ or $(n_1, n_2,\dots, n_k)$ where he will be surrounded. 
\end{proof}

\section{Further Directions}\label{sec: further directions}

In Section \ref{sec: push num}, we gave some results on the push number and showed that $\sigma(G) \leq \ccr(G) + \push(G)$ for any graph $G$. In \cite{HN21}, Huggan and Nowakowski ask whether it is true that $0\leq \sigma(G) - \ccr(G) \leq 1$ for any $G$. One way to answer this in the affirmative is to show that $\push(G) \leq 1$ for any $G$ and use Theorem \ref{thm: bounds on surr. num in terms of ccr and push num}. We have yet to find an example where $\push(G) > 1$. Is it true that $\push(G) \leq 1$ for all $G$?

We could also define an analogous parameter to the push number for the original Cops and Robber game. Consider a \emph{push in Cops and Robber} to be when a cop moves into the robber's neighbourhood, which forces the robber to move to a new vertex or else he will lose on the next cop turn. Define the Cops and Robber push number to be the minimum number of cops, out of $c(G)$ cops playing, needed to push the robber in order to capture him. Let $p_c(G)$ denote this new push number on the graph $G$. This push number can be used to compare Cops and Robber to one if its variants, Cops and Attacking Robbers. In Cops and Attacking Robbers, first introduced in \cite{BFG13} and further studied in \cite{CHM24}, the robber is able to move onto a vertex occupied by a single cop and eliminate that cop from the game. The minimum number of cops needed to win Cops and Attacking Robbers on $G$ is denoted $cc(G)$. Using a similar argument as in the proof of Theorem \ref{thm: bounds on surr. num in terms of ccr and push num}, it can be seen that $cc(G) \leq c(G) + p_c(G)$. 

Notably, there are examples where $p_c(G) > 1$. In \cite{BFG13}, an example was given where $cc(G) - c(G) = 2$. Applying this to the above upper bound, and we get that $p_c(G) \geq 2$. Similarly, in \cite{CHM24} an example is given where $cc(G) - c(G) = 3$ and so, consequently, $p_c(G) \geq 3$. How large can $p_c(G)$ be? What are the characterizations for $p_c(G) = k$ where $k\geq 0$? What other information about Cops and Robber and its variants can be obtained by these push number parameters?

In Section \ref{sec: planar graphs} we gave bounds on the cheating robot number for planar graphs and bipartite planar graphs. While we gave an example showing that the bound $\ccr(G) \leq 4$ is tight for bipartite planar graphs, the question of how large the cheating robot number can get for planar graphs remains open. Since there exist 5-regular planar graphs, for example the icosahedron, by Theorem \ref{thm: k-core lower bound on ccr} and Corollary \ref{cor: ccr of planar}, $5\leq \max\{\ccr(G) | \text{$G$ is planar}\} \leq 7$.

\end{document}